\documentclass[10pt]{amsart}

\usepackage{amsmath}
\usepackage{dsfont}
\usepackage{mathrsfs}
\usepackage{amssymb}
\usepackage{latexsym,amsfonts,amsthm,geometry}
\usepackage{url}
\usepackage{booktabs}

\usepackage{color}
\usepackage{tikz}

\allowdisplaybreaks[3]

\usepackage{multirow}

\usepackage{enumitem}

\usepackage{bm}
\newtheorem{theorem}{Theorem}[section]
\newtheorem{lemma}[theorem]{Lemma}

\newtheorem{corollary}[theorem]{Corollary}
\newtheorem{proposition}[theorem]{Proposition}
\theoremstyle{definition}
\newtheorem{definition}[theorem]{Definition}
\theoremstyle{remark}
\newtheorem{remark}[theorem]{Remark}

\newcommand{\bu}{{\bm u}}
\newcommand{\bv}{{\bm v}}
\newcommand{\bbR}{\mathbb{R}}
\newcommand{\bcurl}{{\bf curl}}
\newcommand{\p}{\partial}
\newcommand{\nab}{\nabla}
\newcommand{\calT}{\mathcal{T}}
\newcommand{\bbP}{\mathbb{P}}

\numberwithin{equation}{section}

\begin{document}

\title[Symmetric counterexample to Strang's conjecture]{
A symmetric counterexample to Strang's conjecture for bivariate $C^1$ cubic splines on triangulations}

\author[Pratyush Potu]{Pratyush Potu}
\address{Division of Applied Mathematics, Brown University, Providence, RI 02912, USA}
\email{pratyush\_potu@brown.edu}
\date{}

\begin{abstract}
We exhibit a triangulation of an equilateral triangle for which the space of bivariate $C^1$ cubic splines has dimension larger than the dimension formula conjectured by Strang. Notably, the triangulation is such that no two edges sharing a vertex are collinear, and the triangulation is invariant under the action of the isometry group ($D_3$) of the equilateral triangle which is triangulated. 
\end{abstract}

\maketitle

\thispagestyle{empty}

\section{Introduction}
Let $\mathcal T$ be a triangulation of a simply connected polygonal domain $\Omega\subset \mathbb R^2$. That is, $\mathcal T$ is a collection of closed triangles $\{T_1, \dots, T_F\}$ such that $\Omega = \bigcup_{i=1}^F T_i$, and if two triangles in $\mathcal T$ intersect, their intersection is either a common vertex or a common edge. The space of bivariate splines of degree $d$ and smoothness $r$ on $\mathcal T$ is
\[
    S_d^r(\mathcal T) := \{s\in C^r(\Omega):s|_T\in \mathbb P_d\:\:\forall T\in \mathcal T\},
\]
where $C^r(\Omega)$ denotes the set of functions defined on $\Omega$ with $r$ continuous derivatives in both variables and $\mathbb{P}_d$ is the set of all bivariate polynomials of total degree at most $d$.

The spaces $S_d^r(\mathcal T)$ have been extensively studied over the past decades due to their extensive applications across applied mathematics in areas such as numerical analysis (particularly finite element methods) and computer aided geometric design. A fundamental question asked by Strang in \cite{S:73, S:74} concerns the dimension of these spaces. For the case $r=1$, which is the focus of this work, Strang proposed the formula:
\begin{equation}\label{eq:strang}
\dim S^1_d(\mathcal T)\ \overset?=\ \tbinom{d+2}{2}F-(2d+1)E_0+3V_0,
\end{equation}
where $F$ denotes the number of triangles, $E_0$ the number of interior edges, and $V_0$ the number of interior vertices. It was soon realized, through an example of Powell discussed in \cite{MS:96}, that the formula must be corrected by including the number of ``singular vertices'' (interior vertices contained by exactly four edges which lie on two straight lines, see Definition~\ref{def:singular}). In particular, Morgan and Scott \cite{MS:75} established that for $d\geq 5$,
\begin{equation}\label{eq:LB}
    \dim S^1_d(\mathcal T) = \tbinom{d+2}{2}F-(2d+1)E_0+3V_0 + \sigma,
\end{equation}
where $\sigma$ is the number of singular vertices in the triangulation. Morgan and Scott \cite{MS:75b} further showed that for $d=2$, the formula \eqref{eq:LB} is not always correct by exhibiting an amazingly simple counterexample comprised of seven triangles for which the true dimension of $S^1_2(\mathcal T)$ is one larger than that predicted by \eqref{eq:LB}. Schumaker \cite{S:79} provided a formula for $\dim S_d^r(\calT)$ for general $r$ and $d$ which reduces to exactly \eqref{eq:LB} when $r=1$ and showed this formula was in fact a lower bound for all $d\geq 2$. Some years later, Alfeld, Piper, and Schumaker \cite{APS:87} showed using so-called Bernstein--B\'ezier techniques that the formula \eqref{eq:LB} is correct for $d=4$. The remaining question for multiple decades was the $d=3$ case. Very recently, this case was closed by Wu and Deng in a breakthrough work \cite{WD:26} which showcased a triangulation of 18 triangles for which the dimension of $S^1_3(\mathcal T)$ exceeds the formula \eqref{eq:LB} by one.

While much of the work mentioned above was carried out using numerical analysis techniques, a separate line of inquiry was initiated by Billera \cite{B:88} who showed that Strang's formula was correct for ``generic'' triangulations by using homological algebra. By applying and extending this foundational work of Billera, many questions have been answered about the spaces $S^r_d(\mathcal T)$ and their generalizations. We refer the interested reader to the beautiful survey of Lanini, Schenck, and Tymoczko \cite{LST:24} and references therein for more details in this direction.

Another reason for interest in the dimension of $S^r_d(\mathcal T)$ arises when one considers the so-called Stokes complex in the study of mixed finite element methods. The dimension of $S^1_{k+1}(\mathcal T)$ is directly related to the exactly divergence-free finite elements of Scott and Vogelius \cite{SV:85,GS:19}. The formula \eqref{eq:LB} holds for $d=k+1$ exactly when the divergence maps the continuous piecewise polynomial vector fields of degree $k$ onto the discontinuous piecewise polynomials of degree $k-1$, modulo the constraints at singular vertices (see Section~\ref{sec:divformulation}). Triangulations where \eqref{eq:LB} does not hold lead to spurious pressure modes in mixed finite element methods which use these piecewise polynomial spaces \cite{AQ:92}.

Turning now to the present work, we note that the mechanism by which the counterexample of Wu and Deng \cite{WD:26} produces a dimension jump relies on two ``quasi-singular vertices'' in the sense of Morgan and Scott \cite[Definition 4]{MS:96}. That is, interior vertices which are not singular but are still contained in two pairs of collinear edges (see Definition~\ref{def:quasisingular}). This is in stark contrast to the Morgan--Scott example, where no two edges at any vertex are collinear and hence the triangulation has no quasi-singular vertices. A natural question thus arises: does a counterexample to the dimension formula \eqref{eq:LB} for $d=3$ exist which also has no quasi-singular vertices? In this work, we answer this question in the affirmative by providing such a triangulation. 

The search for this triangulation relied heavily on the use of the Large Language Model (LLM) Claude Fable 5.1. Naturally, the search space for this problem is quite large. However, by leveraging results which allow us to exclude certain triangulations from being counterexamples to the dimension formula \eqref{eq:LB}, the search process could be made much more efficient. Relevant details on how the search was set up are provided in Section \ref{sec:search_procedure}.

\section{Preliminaries}

\subsection{Triangulations and notation}
Throughout, $\calT$ is a triangulation of a simply connected polygonal domain $\Omega\subset\bbR^2$ as in the introduction. We write $V$, $E$ and $F$ for the numbers of vertices, edges and triangles (faces) of $\calT$ respectively, and $V_0$, $E_0$ for the numbers of interior vertices and interior edges respectively. A vertex is interior if it does not lie on $\p\Omega$ while an edge is interior if it is shared by two triangles. 

For a vertex $v$ we denote the \emph{degree} of $v$ by $\deg v$ which is the number of edges of $\calT$ containing $v$. We denote by $\operatorname{star}(v)$ the union of the triangles containing $v$. The \emph{link} of $v$ is the set of edges of $\operatorname{star}(v)$ not containing $v$. Two edges of $\calT$ with the common endpoint $v$ are called \emph{collinear} if they lie on a common straight line. For any three points $a,b,c$ we write $[a,b,c]:=\det(b-a,\,c-a)$ so that $[a,b,c]=0$ if and only if $a,b,c$ are collinear.

\subsection{Singular and quasi-singular vertices}
The following two kinds of interior vertices are of interest in this work.

\begin{definition}[singular vertex]\label{def:singular}
An interior vertex $v$ of $\calT$ is \emph{singular} if $\deg v=4$ and the four edges at $v$ lie on two straight lines. We write $\sigma=\sigma(\calT)$ for the number of singular vertices of $\calT$.
\end{definition}

\begin{figure}[htb]\centering
\begin{tikzpicture}[scale=1.15]
  \coordinate (v) at (0,0);
  \coordinate (u1) at (15:1.5);
  \coordinate (u2) at (110:1.2);
  \coordinate (u3) at (195:1.3);
  \coordinate (u4) at (290:1.4);
  \draw[gray!55,dashed] (195:1.85) -- (15:2.05);
  \draw[gray!55,dashed] (290:1.9) -- (110:1.7);
  \draw[thin] (u1) -- (u2) -- (u3) -- (u4) -- cycle;
  \draw[very thick] (u3) -- (v) -- (u1);
  \draw[very thick] (u4) -- (v) -- (u2);
  \fill (v) circle (1.7pt);
  \node[below right=-1pt] at (v) {$v$};
  \foreach \u in {u1,u2,u3,u4} \fill (\u) circle (1.3pt);
\end{tikzpicture}
\caption{A singular vertex: $v$ is interior, $\deg v=4$, and the four edges at $v$ (thick) lie on two straight lines (dashed).}\label{fig:singular}
\end{figure}
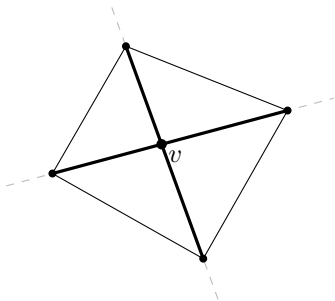

Singular vertices are exactly the vertices involved in the dimension formula \eqref{eq:LB}. Morgan and Scott \cite{MS:75b,MS:96} also identified a second type of degenerate vertices, which they referred to as quasi-singular vertices. Such vertices are not involved in the formula \eqref{eq:LB}, but they obstruct the inductive construction of nodal bases \cite{MS:96}. 

\begin{definition}[quasi-singular vertex, {\cite[Definition 4]{MS:96}}]\label{def:quasisingular}
Let $v$ be an interior vertex of $\calT$ and let $T$ be a triangle of $\calT$ with vertex $v$ and edges $e_1,e_2$
at $v$. Let $T_1$ and $T_2$ be the other triangles of $\calT$ containing $e_1$ and $e_2$, respectively, and let
$e_3\subset T_1$ and $e_4\subset T_2$ be their other edges at $v$. The triangle $T$ is \emph{quasi-singular at
$v$} if the pairs $e_1,e_4$ and $e_2,e_3$ are both collinear and $v$ is not singular. The vertex $v$ is
\emph{quasi-singular} if some triangle of $\calT$ is quasi-singular at $v$.
\end{definition}

In particular, any quasi-singular vertex must be such that $\deg v\ge5$. The counterexample of Wu and Deng \cite{WD:26} involves two of these quasi-singular vertices. Similarly, Yuan and Stillman's \cite{SY:19} counterexample to Schenck and Stiller's more general ``$2r+1$'' conjecture concerning whether Schumaker\'s lower bound \cite{S:79} is obtained for smoothness $r$ when $d=2r+1$ also uses two quasi-singular vertices.

We say that $\calT$ \emph{has no collinear edge pairs} if no two edges of $\calT$ with a common endpoint are collinear. This condition, in particular, excludes singular vertices and quasi-singular vertices. The Morgan--Scott triangulation has this property, and so does the triangulation defined in Section~\ref{sec:triangulation} below.

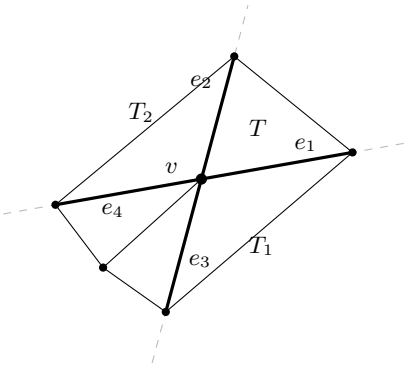
\begin{figure}[htb]\centering
\begin{tikzpicture}[scale=1.4]
  \coordinate (v) at (0,0);
  \coordinate (u1) at (10:1.45);
  \coordinate (u2) at (75:1.2);
  \coordinate (u4) at (190:1.4);
  \coordinate (u5) at (222:1.25);
  \coordinate (u3) at (255:1.3);
  \draw[gray!55,dashed] (190:1.9) -- (10:1.95);
  \draw[gray!55,dashed] (255:1.8) -- (75:1.7);
  \draw[thin] (u1) -- (u2) -- (u4) -- (u5) -- (u3) -- cycle;
  \draw[thin] (v) -- (u5);
  \draw[very thick] (u4) -- (v) -- (u1);
  \draw[very thick] (u3) -- (v) -- (u2);
  \fill (v) circle (1.5pt);
  \node[font=\small] at (160:0.3) {$v$};
  \foreach \u in {u1,u2,u3,u4,u5} \fill (\u) circle (1.1pt);
  \node[font=\small] at (42:0.72) {$T$};
  \node[font=\small] at (312:0.85) {$T_1$};
  \node[font=\small] at (132:0.85) {$T_2$};
  \node[font=\small,above] at (10:1.0) {$e_1$};
  \node[font=\small,left=2pt] at (75:0.95) {$e_2$};
  \node[font=\small,right] at (255:0.8) {$e_3$};
  \node[font=\small,below] at (190:0.85) {$e_4$};
\end{tikzpicture}
\caption{A quasi-singular vertex in the sense of Definition~\ref{def:quasisingular} (cf.\ \cite[Figure 8]{MS:96}).}\label{fig:quasi}
\end{figure}

\subsection{Schumaker's lower bound}
Schumaker's lower bound on $\dim S^1_3(\calT)$ \cite{S:79} is denoted
\begin{equation}\label{eq:LB3}
\operatorname{LB}(\calT):=10F-7E_0+3V_0+\sigma.
\end{equation}
For the triangulation, $\widehat{\mathcal T}$, which we define in the next section, we show that $\dim S^1_3(\widehat{\mathcal T}) > \operatorname{LB}(\widehat{\mathcal T})$.

\section{The triangulation}\label{sec:triangulation}

Let $\Delta$ be the equilateral triangle with vertices
\[
q_1=(-1,0),\qquad q_2=(1,0),\qquad q_3=(0,\sqrt3),
\]
and center $c=(0,\sqrt3/3)$. We write $\rho$ for the rotation by $2\pi/3$ about $c$ that maps $q_1\mapsto q_2\mapsto q_3\mapsto q_1$ and $\mu$ for the reflection in the line through $q_1$ and $c$. Together, $\rho$ and $\mu$ generate the isometry group of $\Delta$ which is the dihedral group of order 6, denoted $D_3$.

\subsection{The main example}

We now state the main result and postpone the proof to Section \ref{sec:verification}.

\begin{theorem}\label{thm:triangulation}
Let $\widehat{\mathcal T}$ be the triangulation of $\Delta$ whose $18$ vertices $v_0,\dots,v_{17}$ are listed in
Table~\ref{tab:vertices} and whose $31$ triangles are
\begin{equation}\label{eq:tris}
\begin{array}{l}
(0,2,1),\ (0,3,2),\ (0,4,3),\ (0,1,4),\ (1,5,4),\ (1,6,5),\ (1,2,6),\ (2,7,6),\ (2,3,7),\\
(3,8,7),\ (3,9,8),\ (3,10,9),\ (4,11,10),\ (4,12,11),\ (4,5,12),\ (5,6,7),\ (5,13,12),\ (5,14,13),\\
(5,7,14),\ (7,15,14),\ (7,8,15),\ (8,16,15),\ (8,9,16),\ (9,10,16),\ (10,17,14),\ (10,14,16),\\
(10,11,17),\ (11,12,17),\ (12,13,17),\ (13,14,17),\ (14,15,16).
\end{array}
\end{equation}
Then:
\begin{enumerate}[label=(\roman*)]
\item $\widehat{\mathcal T}$ is a triangulation of $\Delta$ with $V=18$, $V_0=15$, $E=48$, $E_0=45$ and $F=31$, and it is invariant under the isometry group of $\Delta$.
\item No two edges of $\widehat{\mathcal T}$ with a common endpoint are collinear. In particular, $\widehat{\mathcal T}$ has neither singular nor quasi-singular vertices.
\item $\dim S^1_3(\widehat{\mathcal T})=41$, whereas $\operatorname{LB}(\widehat{\mathcal T})=10F-7E_0+3V_0=40$.
\end{enumerate}
\end{theorem}

\begin{table}[htb]
\centering
\renewcommand{\arraystretch}{1.5}
\begin{tabular}{@{}cc|cc|cc@{}}\hline
Vertex & Coordinates & Vertex & Coordinates & Vertex & Coordinates \\\hline
$v_0$ & $(-\frac{3}{8},\frac{11}{24}\sqrt3)$ & $v_6$    & $(-\frac{1}{4},\frac{5}{12}\sqrt3)$  & $v_{12}$ & $(\frac{1}{4},\frac{17}{36}\sqrt3)$ \\\hline
$v_1$ & $(-\frac{1}{4},\frac{17}{36}\sqrt3)$ & $v_7$    & $(-\frac{1}{3},\frac{2}{9}\sqrt3)$   & $v_{13}$ & $(\frac{1}{4},\frac{5}{12}\sqrt3)$  \\\hline
$v_2$ & $(-\frac{1}{3},\frac{7}{18}\sqrt3)$  & $v_8$    & $(-\frac{1}{12},\frac{5}{36}\sqrt3)$ & $v_{14}$ & $(\frac{1}{3},\frac{2}{9}\sqrt3)$   \\\hline
$v_3$ & $(-1,0)$                             & $v_9$    & $(0,\frac{1}{12}\sqrt3)$             & $v_{15}$ & $(0,\frac{1}{6}\sqrt3)$             \\\hline
$v_4$ & $(0,\sqrt3)$                         & $v_{10}$ & $(1,0)$                              & $v_{16}$ & $(\frac{1}{12},\frac{5}{36}\sqrt3)$ \\\hline
$v_5$ & $(0,\frac{5}{9}\sqrt3)$              & $v_{11}$ & $(\frac{3}{8},\frac{11}{24}\sqrt3)$  & $v_{17}$ & $(\frac{1}{3},\frac{7}{18}\sqrt3)$  \\
\hline
\end{tabular}
\vspace{0.5em}
\caption{The vertices of $\widehat{\calT}$. The corners of $\Delta$ are $v_3=q_1$, $v_{10}=q_2$ and $v_4=q_3$.}
\label{tab:vertices}
\end{table}

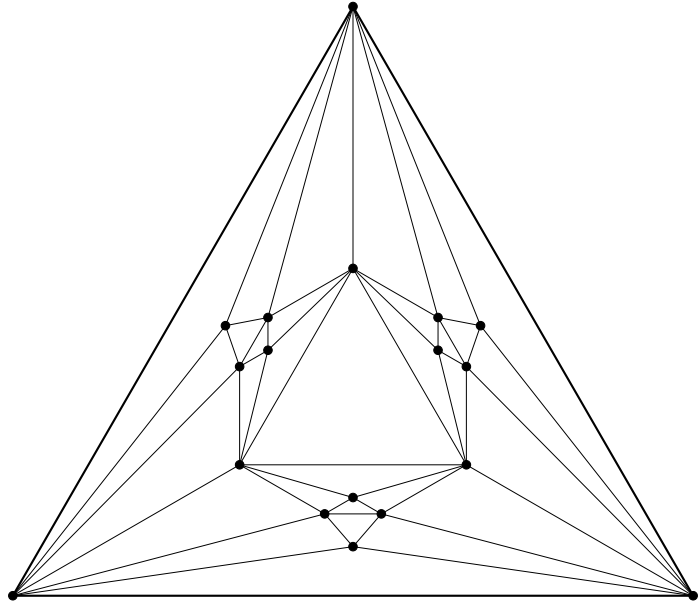
\begin{figure}[htb]\centering
\begin{tikzpicture}[scale=4.500]
\draw[line width=0.3pt] (-0.3750,0.7939) -- (-0.2500,0.8179);
\draw[line width=0.3pt] (-0.3750,0.7939) -- (-0.3333,0.6736);
\draw[line width=0.3pt] (-0.3750,0.7939) -- (-1.0000,0.0000);
\draw[line width=0.3pt] (-0.3750,0.7939) -- (0.0000,1.7321);
\draw[line width=0.3pt] (-0.2500,0.8179) -- (-0.3333,0.6736);
\draw[line width=0.3pt] (-0.2500,0.8179) -- (0.0000,1.7321);
\draw[line width=0.3pt] (-0.2500,0.8179) -- (0.0000,0.9623);
\draw[line width=0.3pt] (-0.2500,0.8179) -- (-0.2500,0.7217);
\draw[line width=0.3pt] (-0.3333,0.6736) -- (-1.0000,0.0000);
\draw[line width=0.3pt] (-0.3333,0.6736) -- (-0.2500,0.7217);
\draw[line width=0.3pt] (-0.3333,0.6736) -- (-0.3333,0.3849);
\draw[line width=0.8pt] (-1.0000,0.0000) -- (0.0000,1.7321);
\draw[line width=0.3pt] (-1.0000,0.0000) -- (-0.3333,0.3849);
\draw[line width=0.3pt] (-1.0000,0.0000) -- (-0.0833,0.2406);
\draw[line width=0.3pt] (-1.0000,0.0000) -- (0.0000,0.1443);
\draw[line width=0.8pt] (-1.0000,0.0000) -- (1.0000,0.0000);
\draw[line width=0.3pt] (0.0000,1.7321) -- (0.0000,0.9623);
\draw[line width=0.8pt] (0.0000,1.7321) -- (1.0000,0.0000);
\draw[line width=0.3pt] (0.0000,1.7321) -- (0.3750,0.7939);
\draw[line width=0.3pt] (0.0000,1.7321) -- (0.2500,0.8179);
\draw[line width=0.3pt] (0.0000,0.9623) -- (-0.2500,0.7217);
\draw[line width=0.3pt] (0.0000,0.9623) -- (-0.3333,0.3849);
\draw[line width=0.3pt] (0.0000,0.9623) -- (0.2500,0.8179);
\draw[line width=0.3pt] (0.0000,0.9623) -- (0.2500,0.7217);
\draw[line width=0.3pt] (0.0000,0.9623) -- (0.3333,0.3849);
\draw[line width=0.3pt] (-0.2500,0.7217) -- (-0.3333,0.3849);
\draw[line width=0.3pt] (-0.3333,0.3849) -- (-0.0833,0.2406);
\draw[line width=0.3pt] (-0.3333,0.3849) -- (0.3333,0.3849);
\draw[line width=0.3pt] (-0.3333,0.3849) -- (0.0000,0.2887);
\draw[line width=0.3pt] (-0.0833,0.2406) -- (0.0000,0.1443);
\draw[line width=0.3pt] (-0.0833,0.2406) -- (0.0000,0.2887);
\draw[line width=0.3pt] (-0.0833,0.2406) -- (0.0833,0.2406);
\draw[line width=0.3pt] (0.0000,0.1443) -- (1.0000,0.0000);
\draw[line width=0.3pt] (0.0000,0.1443) -- (0.0833,0.2406);
\draw[line width=0.3pt] (1.0000,0.0000) -- (0.3750,0.7939);
\draw[line width=0.3pt] (1.0000,0.0000) -- (0.3333,0.3849);
\draw[line width=0.3pt] (1.0000,0.0000) -- (0.0833,0.2406);
\draw[line width=0.3pt] (1.0000,0.0000) -- (0.3333,0.6736);
\draw[line width=0.3pt] (0.3750,0.7939) -- (0.2500,0.8179);
\draw[line width=0.3pt] (0.3750,0.7939) -- (0.3333,0.6736);
\draw[line width=0.3pt] (0.2500,0.8179) -- (0.2500,0.7217);
\draw[line width=0.3pt] (0.2500,0.8179) -- (0.3333,0.6736);
\draw[line width=0.3pt] (0.2500,0.7217) -- (0.3333,0.3849);
\draw[line width=0.3pt] (0.2500,0.7217) -- (0.3333,0.6736);
\draw[line width=0.3pt] (0.3333,0.3849) -- (0.0000,0.2887);
\draw[line width=0.3pt] (0.3333,0.3849) -- (0.0833,0.2406);
\draw[line width=0.3pt] (0.3333,0.3849) -- (0.3333,0.6736);
\draw[line width=0.3pt] (0.0000,0.2887) -- (0.0833,0.2406);
\filldraw[fill=black] (-0.3750,0.7939) circle (0.0125cm);
\filldraw[fill=black] (-0.2500,0.8179) circle (0.0125cm);
\filldraw[fill=black] (-0.3333,0.6736) circle (0.0125cm);
\filldraw[fill=black] (-1.0000,0.0000) circle (0.0125cm);
\filldraw[fill=black] (0.0000,1.7321) circle (0.0125cm);
\filldraw[fill=black] (0.0000,0.9623) circle (0.0125cm);
\filldraw[fill=black] (-0.2500,0.7217) circle (0.0125cm);
\filldraw[fill=black] (-0.3333,0.3849) circle (0.0125cm);
\filldraw[fill=black] (-0.0833,0.2406) circle (0.0125cm);
\filldraw[fill=black] (0.0000,0.1443) circle (0.0125cm);
\filldraw[fill=black] (1.0000,0.0000) circle (0.0125cm);
\filldraw[fill=black] (0.3750,0.7939) circle (0.0125cm);
\filldraw[fill=black] (0.2500,0.8179) circle (0.0125cm);
\filldraw[fill=black] (0.2500,0.7217) circle (0.0125cm);
\filldraw[fill=black] (0.3333,0.3849) circle (0.0125cm);
\filldraw[fill=black] (0.0000,0.2887) circle (0.0125cm);
\filldraw[fill=black] (0.0833,0.2406) circle (0.0125cm);
\filldraw[fill=black] (0.3333,0.6736) circle (0.0125cm);
\end{tikzpicture}
\caption{The triangulation $\widehat{\mathcal T}$ of Theorem~\ref{thm:triangulation}.}\label{fig:triangulation}
\end{figure}

\begin{figure}[htb]\centering
\begin{tikzpicture}[scale=8.500]
\clip (-0.520,0.080) rectangle (0.520,1.000);
\draw[line width=0.3pt] (-0.3750,0.7939) -- (-0.2500,0.8179);
\draw[line width=0.3pt] (-0.3750,0.7939) -- (-0.3333,0.6736);
\draw[line width=0.3pt] (-0.3750,0.7939) -- (-1.0000,0.0000);
\draw[line width=0.3pt] (-0.3750,0.7939) -- (0.0000,1.7321);
\draw[line width=0.3pt] (-0.2500,0.8179) -- (-0.3333,0.6736);
\draw[line width=0.3pt] (-0.2500,0.8179) -- (0.0000,1.7321);
\draw[line width=0.3pt] (-0.2500,0.8179) -- (0.0000,0.9623);
\draw[line width=0.3pt] (-0.2500,0.8179) -- (-0.2500,0.7217);
\draw[line width=0.3pt] (-0.3333,0.6736) -- (-1.0000,0.0000);
\draw[line width=0.3pt] (-0.3333,0.6736) -- (-0.2500,0.7217);
\draw[line width=0.3pt] (-0.3333,0.6736) -- (-0.3333,0.3849);
\draw[line width=0.8pt] (-1.0000,0.0000) -- (0.0000,1.7321);
\draw[line width=0.3pt] (-1.0000,0.0000) -- (-0.3333,0.3849);
\draw[line width=0.3pt] (-1.0000,0.0000) -- (-0.0833,0.2406);
\draw[line width=0.3pt] (-1.0000,0.0000) -- (0.0000,0.1443);
\draw[line width=0.8pt] (-1.0000,0.0000) -- (1.0000,0.0000);
\draw[line width=0.3pt] (0.0000,1.7321) -- (0.0000,0.9623);
\draw[line width=0.8pt] (0.0000,1.7321) -- (1.0000,0.0000);
\draw[line width=0.3pt] (0.0000,1.7321) -- (0.3750,0.7939);
\draw[line width=0.3pt] (0.0000,1.7321) -- (0.2500,0.8179);
\draw[line width=0.3pt] (0.0000,0.9623) -- (-0.2500,0.7217);
\draw[line width=0.3pt] (0.0000,0.9623) -- (-0.3333,0.3849);
\draw[line width=0.3pt] (0.0000,0.9623) -- (0.2500,0.8179);
\draw[line width=0.3pt] (0.0000,0.9623) -- (0.2500,0.7217);
\draw[line width=0.3pt] (0.0000,0.9623) -- (0.3333,0.3849);
\draw[line width=0.3pt] (-0.2500,0.7217) -- (-0.3333,0.3849);
\draw[line width=0.3pt] (-0.3333,0.3849) -- (-0.0833,0.2406);
\draw[line width=0.3pt] (-0.3333,0.3849) -- (0.3333,0.3849);
\draw[line width=0.3pt] (-0.3333,0.3849) -- (0.0000,0.2887);
\draw[line width=0.3pt] (-0.0833,0.2406) -- (0.0000,0.1443);
\draw[line width=0.3pt] (-0.0833,0.2406) -- (0.0000,0.2887);
\draw[line width=0.3pt] (-0.0833,0.2406) -- (0.0833,0.2406);
\draw[line width=0.3pt] (0.0000,0.1443) -- (1.0000,0.0000);
\draw[line width=0.3pt] (0.0000,0.1443) -- (0.0833,0.2406);
\draw[line width=0.3pt] (1.0000,0.0000) -- (0.3750,0.7939);
\draw[line width=0.3pt] (1.0000,0.0000) -- (0.3333,0.3849);
\draw[line width=0.3pt] (1.0000,0.0000) -- (0.0833,0.2406);
\draw[line width=0.3pt] (1.0000,0.0000) -- (0.3333,0.6736);
\draw[line width=0.3pt] (0.3750,0.7939) -- (0.2500,0.8179);
\draw[line width=0.3pt] (0.3750,0.7939) -- (0.3333,0.6736);
\draw[line width=0.3pt] (0.2500,0.8179) -- (0.2500,0.7217);
\draw[line width=0.3pt] (0.2500,0.8179) -- (0.3333,0.6736);
\draw[line width=0.3pt] (0.2500,0.7217) -- (0.3333,0.3849);
\draw[line width=0.3pt] (0.2500,0.7217) -- (0.3333,0.6736);
\draw[line width=0.3pt] (0.3333,0.3849) -- (0.0000,0.2887);
\draw[line width=0.3pt] (0.3333,0.3849) -- (0.0833,0.2406);
\draw[line width=0.3pt] (0.3333,0.3849) -- (0.3333,0.6736);
\draw[line width=0.3pt] (0.0000,0.2887) -- (0.0833,0.2406);
\filldraw[fill=blue!70!black] (-0.3750,0.7939) circle (0.020cm);
\node[font=\scriptsize, inner sep=0.5pt] at (-0.4044,0.8109) {0};
\filldraw[fill=white] (-0.2500,0.8179) circle (0.020cm);
\node[font=\scriptsize, inner sep=0.5pt] at (-0.2732,0.8428) {1};
\filldraw[fill=white] (-0.3333,0.6736) circle (0.020cm);
\node[font=\scriptsize, inner sep=0.5pt] at (-0.3662,0.6824) {2};
\filldraw[fill=red!75!black] (0.0000,0.9623) circle (0.020cm);
\node[font=\scriptsize, inner sep=0.5pt] at (-0.0294,0.9793) {5};
\filldraw[fill=blue!70!black] (-0.2500,0.7217) circle (0.020cm);
\node[font=\scriptsize, inner sep=0.5pt] at (-0.2206,0.7047) {6};
\filldraw[fill=red!75!black] (-0.3333,0.3849) circle (0.020cm);
\node[font=\scriptsize, inner sep=0.5pt] at (-0.3628,0.4019) {7};
\filldraw[fill=white] (-0.0833,0.2406) circle (0.020cm);
\node[font=\scriptsize, inner sep=0.5pt] at (-0.0933,0.2080) {8};
\filldraw[fill=blue!70!black] (0.0000,0.1443) circle (0.020cm);
\node[font=\scriptsize, inner sep=0.5pt] at (-0.0000,0.1103) {9};
\filldraw[fill=blue!70!black] (0.3750,0.7939) circle (0.020cm);
\node[font=\scriptsize, inner sep=0.5pt] at (0.4044,0.8109) {11};
\filldraw[fill=white] (0.2500,0.8179) circle (0.020cm);
\node[font=\scriptsize, inner sep=0.5pt] at (0.2732,0.8428) {12};
\filldraw[fill=blue!70!black] (0.2500,0.7217) circle (0.020cm);
\node[font=\scriptsize, inner sep=0.5pt] at (0.2206,0.7047) {13};
\filldraw[fill=red!75!black] (0.3333,0.3849) circle (0.020cm);
\node[font=\scriptsize, inner sep=0.5pt] at (0.3628,0.4019) {14};
\filldraw[fill=blue!70!black] (0.0000,0.2887) circle (0.020cm);
\node[font=\scriptsize, inner sep=0.5pt] at (0.0000,0.3227) {15};
\filldraw[fill=white] (0.0833,0.2406) circle (0.020cm);
\node[font=\scriptsize, inner sep=0.5pt] at (0.0933,0.2080) {16};
\filldraw[fill=white] (0.3333,0.6736) circle (0.020cm);
\node[font=\scriptsize, inner sep=0.5pt] at (0.3662,0.6824) {17};
\end{tikzpicture}
\caption{Close-up of the interior of $\widehat{\calT}$, with the vertex numbering of Table~\ref{tab:vertices}. The certices of the central triangle $(5,14,7)$ (red), the other six vertices on the three symmetry axes (blue), and the six vertices of the regular orbit (white).}\label{fig:closeup}
\end{figure}
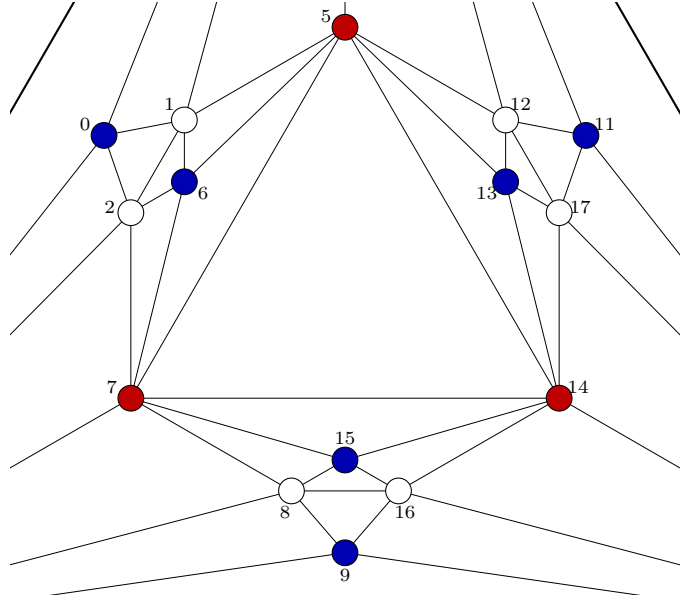

\begin{remark}[Symmetries]
    The interior vertices fall into four orbits of the isometry group of $\Delta$, $D_3$, noted above. The central triangle $\{v_5,v_{14},v_7\}$ consists of three vertices of degree seven which lie on the three symmetry axes of $\Delta$. Two further orbits are $\{v_0,v_{11},v_9\}$ and $\{v_6,v_{13},v_{15}\}$ which consist of vertices of degree four on the axes. Lastly, we have the regular orbit $\{v_1,v_{17},v_8,v_{16},v_{12},v_2\}$ of vertices of degree five. 
\end{remark}

\begin{remark}[integer coordinates]\label{rem:integer}
The dimension of $S^r_d$ is invariant under affine maps of the plane, and the affine map
\[
\Psi(p,\,q\sqrt3)=\big(36(p-q)+12,\ 72q-24\big)
\]
sends $q_1,q_2,q_3$ to $(-24,-24)$, $(48,-24)$, $(-24,48)$ and every vertex of $\widehat{\mathcal T}$ to a point with integer
coordinates, namely
\begin{center}
\renewcommand{\arraystretch}{1.5}
\begin{tabular}{@{}cc|cc|cc@{}}\hline
Vertex & Coordinates & Vertex & Coordinates & Vertex & Coordinates \\\hline
$v_0$ & $(-18,9)$   & $v_6$    & $(-12,6)$   & $v_{12}$ & $(4,10)$   \\\hline
$v_1$ & $(-14,10)$  & $v_7$    & $(-8,-8)$   & $v_{13}$ & $(6,6)$    \\\hline
$v_2$ & $(-14,4)$   & $v_8$    & $(4,-14)$   & $v_{14}$ & $(16,-8)$  \\\hline
$v_3$ & $(-24,-24)$ & $v_9$    & $(9,-18)$   & $v_{15}$ & $(6,-12)$  \\\hline
$v_4$ & $(-24,48)$  & $v_{10}$ & $(48,-24)$  & $v_{16}$ & $(10,-14)$ \\\hline
$v_5$ & $(-8,16)$   & $v_{11}$ & $(9,9)$     & $v_{17}$ & $(10,4)$   \\
\hline
\end{tabular}
\end{center}
In particular, $\Psi(\widehat{\mathcal T})$ is a triangulation of the right triangle with corners $(-24,-24)$, $(48,-24)$, $(-24,48)$ which is also a counterexample to the dimension formula \eqref{eq:LB} with the same dimension $41$. 
\end{remark}

\section{Verification of the main result}\label{sec:verification}

We now prove Theorem~\ref{thm:triangulation}. One can see statement (i) of Theorem~\ref{thm:triangulation} by computing the determinants $[a,b,c]$ for any set of three vertices $a,b,c$ which form a triangle in $\widehat{\mathcal{T}}$. One can see statement (ii) of Theorem~\ref{thm:triangulation} by observing the isometries of $\Delta$ ($\rho$ and $\mu$) map vertices of $\widehat{\mathcal{T}}$ to vertices of $\widehat{\mathcal{T}}$. It remains to compute $\dim S_3^1(\widehat{\mathcal T})$. This can be done in different ways. We provide explicit details for one computation, thus showing statement (iii) of Theorem~\ref{thm:triangulation}, and note two other ways of verifying this computation.

\subsection{The divergence formulation}\label{sec:divformulation}

The dimension of $S^1_3(\widehat{\mathcal{T}})$ can be computed from a discrete Stokes complex. We refer the reader to \cite{SV:85,FN:13} for this complex and its role in the study of finite element methods for Stokes' equations. In fact, this formulation was what was used to find and test the triangulation of Theorem~\ref{thm:triangulation} (see Section \ref{sec:search_procedure}). Let
\begin{align*}
\bm V_h&:=\{\bv\in [C^0(\Omega)]^2:\ \bv|_T\in[\bbP_2]^2\ \text{ for all }T\in\widehat{\mathcal{T}}\},\\
Q_h&:=\{q\in L^2(\Omega):\ q|_T\in\bbP_1\ \text{ for all }T\in\widehat{\mathcal{T}}\},
\end{align*}
be the space of continuous piecewise quadratic vector fields (the vector $P_2$ Lagrange finite element space \cite{EG:21}, which has dimension $2(V+E)$) and the space of discontinuous piecewise linear functions (which has dimension $3F$) respectively. For $s\in S^1_3(\widehat{\mathcal{T}})$, the rotated gradient $\bcurl\,s:=(\p_ys,-\p_xs)$ is continuous, piecewise quadratic and divergence free.

\begin{lemma}\label{lem:complex}
Let $\calT$ be a triangulation of a simply connected domain $\Omega$. The sequence
\begin{equation}\label{eq:complex}
0\longrightarrow\bbR\longrightarrow S^1_3(\calT)\xrightarrow{\ \bcurl\ }\bm V_h\xrightarrow{\ \operatorname{div}\ }Q_h
\end{equation}
is exact at $S^1_3(\calT)$ and at $\bm V_h$. In particular,
\begin{equation}\label{eq:dimdiv}
\dim S^1_3(\calT)=2(V+E)+1-\operatorname{rank}\big(\operatorname{div}:\bm V_h\to Q_h\big).
\end{equation}
\end{lemma}
\begin{proof}
The kernel of $\bcurl$ consists of the constants since $\Omega$ is connected. If $\bv=(v_1,v_2)\in\bm V_h$ satisfies $\operatorname{div}\bv=0$, then $\bu:=(-v_2,v_1)$ is a continuous piecewise polynomial vector field with $\p_xu_2-\p_yu_1=\operatorname{div}\bv=0$. Hence, $\bu$ is curl free in the sense of distributions and, $\bu=\nab s$ for a function $s$ in the Sobolev space $H^2(\Omega)$ since $\Omega$ is simply connected. On each triangle $T\in\mathcal{T}$, $\nab s$ is a quadratic polynomial, so $s|_T\in\bbP_3$. Since $s$ and $\nab s=\bu$ are continuous, we have $s\in S^1_3(\calT)$ and $\bv=\bcurl\,s$. Thus $\ker(\operatorname{div})=\bcurl\,S^1_3(\calT)\cong S^1_3(\calT)/\bbR$, and \eqref{eq:dimdiv} follows from $\dim\bm V_h=2(V+E)$. 
\end{proof}

From the Euler relations (see e.g. \cite[Remark 8.13]{EG:21}), and the fact that $\Omega=\Delta$ is simply connected, we have that $V-E+F = 1$ and $E-E_0 = V-V_0$. Counting edge-triangle incidences gives $2E_0 + (E-E_0) = 3F$. From these three relations one obtains
$2V+2E+1-3F=10F-7E_0+3V_0$, so that \eqref{eq:dimdiv} and \eqref{eq:LB3} combine to
\begin{equation}\label{eq:equiv}
\dim S^1_3(\calT)=\operatorname{LB}(\calT)\quad\Longleftrightarrow\quad
\operatorname{rank}\big(\operatorname{div}:\bm V_h\to Q_h\big)=3F-\sigma ,
\end{equation}
and in general $\dim S^1_3(\calT)-\operatorname{LB}(\calT)=\dim\operatorname{coker}(\operatorname{div})-\sigma$. The number $\sigma$ in \eqref{eq:equiv} is explained by an observation of Scott and Vogelius \cite{SV:85}: if $v$ is a singular vertex with triangles $T_1,\dots,T_4$ in cyclic order, then every $q=\operatorname{div}\bv$ with $\bv\in\bm V_h$ satisfies $\sum_{i=1}^4(-1)^i\,q|_{T_i}(v)=0$, so that $\operatorname{div}\bm V_h$ has codimension at least $\sigma$ in $Q_h$. 

Thus, the dimension formula \eqref{eq:LB} holds for $d=3$ exactly when the singular-vertex conditions are the only constraints on the range of the divergence. In particular, on a triangulation without singular vertices, the dimension formula \eqref{eq:LB} holds for $d=3$ exactly when $\operatorname{div}:\bm V_h\to Q_h$ is surjective. In the language of mixed finite elements, this says that the finite element pair $\bm V_h\times Q_h$  has no spurious pressure modes other than those at singular vertices (cf. \cite{SV:85,GS:19,JLMNR:17}). If a triangulation is such that \eqref{eq:LB} fails, then there must exist a nonzero $p\in Q_h$, not corresponding to the singular vertex constraints, with $\int_\Omega p\operatorname{div}\bv=0$ for all $\bv\in\bm V_h$.

As the spaces $\bm V_h$ and $Q_h$ are finite dimensional, once we fix bases on them, we can directly compute $\operatorname{rank}\big(\operatorname{div}:\bm V_h\to Q_h\big)$ and hence $\dim S_3^1(\calT)$. We do this now, and then show exactly the spurious pressure mode for the $\bm V_h\times Q_h$ pair on $\widehat{\calT}$. 

For a positively oriented triangle $T=(a,b,c)$ with barycentric coordinates $\lambda_a,\lambda_b,\lambda_c$ put $G_a:=(y_b-y_c,\,x_c-x_b)$, so that $\nab\lambda_a=G_a/[a,b,c]$, and cyclically for $G_b,G_c$. The $P_2$ Lagrange basis functions are $\varphi_a=\lambda_a(2\lambda_a-1)$ at the vertices and $\varphi_{ab}=4\lambda_a\lambda_b$ at the edge midpoints, and in the basis $(\lambda_a,\lambda_b,\lambda_c)$ of $\bbP_1(T)$, their derivatives are
\begin{equation}\label{eq:divbasis}
[a,b,c]\,\p_i\varphi_a=G_{a,i}\,(3\lambda_a-\lambda_b-\lambda_c),\qquad
[a,b,c]\,\p_i\varphi_{ab}=4\,(G_{b,i}\lambda_a+G_{a,i}\lambda_b)\qquad(i=1,2).
\end{equation}

We denote by $\mathrm{Div}$ the $3F\times2(V+E)$ matrix which corresponds to the action of $\operatorname{div}:\bm V_h\to Q_h$ with respect to the vector Lagrange basis of $\bm V_h$ and the bases $(\lambda_a,\lambda_b,\lambda_c)$ of the spaces $\bbP_1(T)$. The three
rows of each triangle $T$ (indexed by the nodal values at the vertices of $T$, in the order of \eqref{eq:tris}) are multiplied by $[T]$ where $[T]=[a,b,c]$ for $T$ having vertices $a$, $b$, and $c$. The columns are ordered by the vertices $v_0,\dots,v_{17}$ and then by the edges in lexicographic order of their endpoint indices, with the two components of each node consecutive. For the triangulation $\widehat{\mathcal{T}}$, the entries of $\mathrm{Div}$ are rational multiples of either $1$ or $\sqrt3$.

\subsection{Proof of Theorem~\ref{thm:triangulation}(iii)}\label{sec:proofiii}

We check the left kernel of the matrix $\mathrm{Div} \in \mathbb{R}^{3F \times 2(V+E)} = \mathbb{R}^{93 \times 132}$.
\begin{proposition}\label{prop:rank}
$\operatorname{rank}\mathrm{Div}=92=3F-1$. Consequently $\dim S^1_3(\calT)=132+1-92=41$.
\end{proposition}

\begin{proof}
Let $w\in\mathbb Z^{93}$ be the vector of Table~\ref{tab:w}. A direct computation shows that $w^{\mathsf T}\mathrm{Div}=0$ exactly. Since $w\ne0$, we have $\operatorname{rank}\mathrm{Div}\le92$.
The reverse inequality follows from Gaussian elimination. The dimension follows from \eqref{eq:dimdiv}.
\end{proof}

\begin{table}[htb]\centering\footnotesize
\setlength{\tabcolsep}{3.5pt}
\renewcommand{\arraystretch}{1.2}
\begin{tabular}{@{}lrrr@{\quad}|@{\quad}lrrr@{\quad}|@{\quad}lrrr@{}}\hline
$T$&$w_{T,a}$&$w_{T,b}$&$w_{T,c}$&$T$&$w_{T,a}$&$w_{T,b}$&$w_{T,c}$&$T$&$w_{T,a}$&$w_{T,b}$&$w_{T,c}$\\\hline
$(0,2,1)$ & $898$ & $-1694$ & $-1694$ & $(3,10,9)$ & $0$ & $0$ & $18$ & $(8,16,15)$ & $-1694$ & $-1694$ & $7378$ \\
$(0,3,2)$ & $-62$ & $-14$ & $322$ & $(4,11,10)$ & $0$ & $18$ & $0$ & $(8,9,16)$ & $-1694$ & $898$ & $-1694$ \\
$(0,4,3)$ & $18$ & $0$ & $0$ & $(4,12,11)$ & $-14$ & $322$ & $-62$ & $(9,10,16)$ & $-62$ & $-14$ & $322$ \\
$(0,1,4)$ & $-62$ & $322$ & $-14$ & $(4,5,12)$ & $7$ & $7$ & $-182$ & $(10,17,14)$ & $7$ & $-182$ & $7$ \\
$(1,5,4)$ & $-182$ & $7$ & $7$ & $(5,6,7)$ & $7$ & $574$ & $7$ & $(10,14,16)$ & $7$ & $7$ & $-182$ \\
$(1,6,5)$ & $1330$ & $-1694$ & $-182$ & $(5,13,12)$ & $-182$ & $-1694$ & $1330$ & $(10,11,17)$ & $-14$ & $-62$ & $322$ \\
$(1,2,6)$ & $-1694$ & $-1694$ & $7378$ & $(5,14,13)$ & $7$ & $7$ & $574$ & $(11,12,17)$ & $898$ & $-1694$ & $-1694$ \\
$(2,7,6)$ & $1330$ & $-182$ & $-1694$ & $(5,7,14)$ & $7$ & $7$ & $7$ & $(12,13,17)$ & $-1694$ & $7378$ & $-1694$ \\
$(2,3,7)$ & $-182$ & $7$ & $7$ & $(7,15,14)$ & $7$ & $574$ & $7$ & $(13,14,17)$ & $-1694$ & $-182$ & $1330$ \\
$(3,8,7)$ & $7$ & $-182$ & $7$ & $(7,8,15)$ & $-182$ & $1330$ & $-1694$ & $(14,15,16)$ & $-182$ & $-1694$ & $1330$ \\
$(3,9,8)$ & $-14$ & $-62$ & $322$ & & & & & & & & \\
\hline
\end{tabular}
\caption{The vector $w$ with $w^{\mathsf T}\mathrm{Div}=0$. For each triangle $T=(a,b,c)$ of \eqref{eq:tris} the three entries $w_{T,a},w_{T,b},w_{T,c}$ belong to its corresponding rows. }\label{tab:w}
\end{table}
Proposition~\ref{prop:rank} proves (iii) of Theorem~\ref{thm:triangulation}. 
In particular, $w$ encodes the linear functional
\[
  \ell(q) := \sum_{T}[T]\sum_{a\in T} w_{T,a}\, q|_T(a), \qquad q\in Q_h,
\]
which is such that $\ell(\operatorname{div} \bv) = 0$ for all $\bv \in \bm V_h$.
Since $Q_h$ is a finite-dimensional subspace of $L^2(\Omega)$, there is a unique $p\in Q_h$, the $L^2$-representative of $\ell$, such that $\int_\Omega p\, q = \ell(q)$ for all $q\in Q_h$.
Using $\int_T\lambda_a\lambda_b = [T](1+\delta_{ab})/24$, one finds that on each triangle $T=(a,b,c)$,
\[
  p|_T(a) = 6\bigl(3w_{T,a}-w_{T,b}-w_{T,c}\bigr),
\]
and analogously at $b$ and $c$. Taking $q=\operatorname{div}\bv$ gives $\int_\Omega p\operatorname{div}\bv = 0$ for every $\bv\in\bm V_h$. Moreover, $p$ is not constant. Thus $p$ is a spurious pressure mode of the pair $\bm V_h\times Q_h$ on a triangulation without singular vertices.
The function $p$ is shown in Figure~\ref{fig:p}.

\begin{figure}[htbp]
  \centering
  \includegraphics[width=0.95\textwidth]{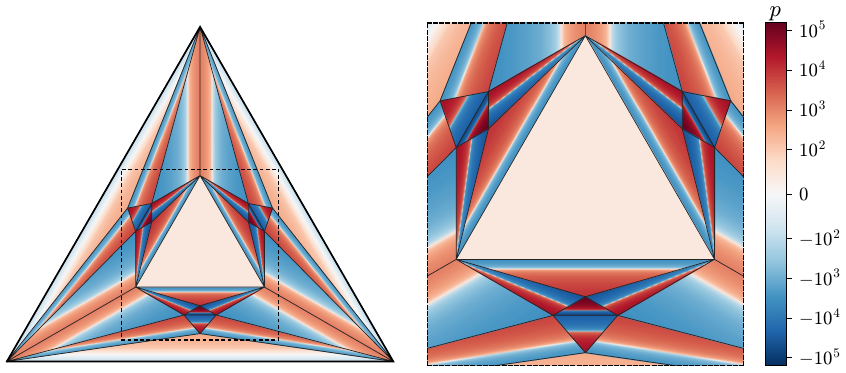}
  \caption{The spurious pressure mode $p$ on $\widehat{\mathcal T}$. It is linear on each triangle and in general discontinuous across edges. It satisfies $\int_\Omega p\operatorname{div}\bv = 0$ for every $\bv\in\bm V_h$.}
  \label{fig:p}
\end{figure}

\subsection{Two other ways to compute the dimension}\label{sec:confirm}
The following two classical descriptions of $S^1_3(\calT)$ lead to different matrices, whose ranks also confirm Proposition~\ref{prop:rank}.

One can compute the dimension using the Bernstein--B\'ezier form as detailed in \cite{LS:07}.  Writing $s|_T=\sum_{i+j+k=3}c^T_{ijk}B^T_{ijk}$ in the Bernstein basis of each triangle, a continuous piecewise cubic is determined by its coefficients at the $V+2E+F=145$ domain points of degree three. The coefficients on a common edge are shared, and $s$ is $C^1$ across the interior edge $e=ab$ shared by $T=(a,b,c)$ and $T^{\prime}=(a,b,d)$ if and only if \cite{F:86}, \cite[Theorem 2.28]{LS:07}
\begin{equation}\label{eq:BB}
c^{T^{\prime}}_{ij1}=\lambda_a(d)\,c^T_{i+1,j,0}+\lambda_b(d)\,c^T_{i,j+1,0}+\lambda_c(d)\,c^T_{ij1},\qquad i+j=2 ,
\end{equation}
where $(\lambda_a(d),\lambda_b(d),\lambda_c(d))=\big([d,b,c],[a,d,c],[a,b,d]\big)/[a,b,c]$ are the barycentric coordinates of $d$ with respect to $T$. 

Multiplying \eqref{eq:BB} by $[a,b,c]$ gives $3E_0=135$ linear conditions on the $145$ coefficients. The exact rank of this $135\times145$ system is $104$ for $\widehat{\mathcal{T}}$, and so $\dim S^1_3(\widehat{\mathcal{T}})=145-104=41$. One can compute the $41$ coefficient vectors that satisfy all $135$ conditions and verify they are linearly independent. This gives an explicit basis of $S^1_3(\widehat{\mathcal{T}})$.

Another way of computing the dimension comes from the homological algebra approach of Billera \cite{B:88}. Consider two polynomials $p_T,p_{T^{\prime}}$ on triangles sharing an edge $e$. Let the affine equation corresponding to $e$ be $\ell_e=0$. Then, the two polynomials are $C^1$ across $e$ if and only if $\ell_e^{\,2}$ divides $p_T-p_{T^{\prime}}$ \cite{B:88,BR:89}. Therefore $S^1_3(\widehat{\mathcal{T}})$ is isomorphic, via $s\mapsto(p_T)_T$, to the space of solutions $\big((p_T)_T,(q_e)_e\big)$ of the linear system
\begin{equation}\label{eq:BR}
p_T-p_{T^{\prime}}=\ell_e^{\,2}\,q_e\qquad(e=T\cap T^{\prime}\ \text{interior}),\qquad p_T\in\bbP_3,\ q_e\in\bbP_1 ,
\end{equation}
the cofactor $q_e$ being determined by $p_T-p_{T^{\prime}}$. For $\widehat{\mathcal{T}}$, the system has $10F+3E_0=445$ unknowns and $10E_0=450$ equations, and its exact rank is $404$, so we again see that $\dim S^1_3(\widehat{\mathcal{T}})=445-404=41$. This formulation is essentially the starting point of the approach of Billera \cite{B:88} and Schenck and Stillman \cite{SS:97a,SS:97b}. It would be interesting to see the full Billera chain complex for the triangulation $\widehat{\mathcal{T}}$ and also understand $S_3^1(\widehat{\mathcal{T}})$ in terms of representations as discussed in \cite[Section 3]{LST:24}.

\begin{remark}\label{rem:quadratic}
The triangulation $\widehat{\mathcal{T}}$ also violates the formula \eqref{eq:LB} for $d=2$: $\dim S^1_2(\widehat{\mathcal{T}})=10$ while $6F-5E_0+3V_0=6$.  
\end{remark}

\section{How this triangulation was found}\label{sec:search_procedure}
We now describe the search that produced $\calT$. We also include some discussion of how the search space could be reduced.

\subsection{Setup of the LLM-based search}
The computations were all carried out by the LLM Claude Fable 5.1, run as an autonomous agent in a sandboxed computing environment over a sequence of sessions. Each session started by reading the original prompt as well as files left over from earlier sessions. Each session ended by updating these files for future sessions. The original prompt consisted of fixing the question, the definitions, the computational method, and the rules of evidence.

In particular, the main question asked to find a triangulation of a topological disk with $\dim S_3^1(\calT) > \operatorname{LB}$ such that no vertex is quasi-singular. The computation was entirely driven by computing the rank of the matrix $\mathrm{Div}$ as outlined in Section \ref{sec:divformulation} for each tested triangulation.

Essentially four strategies were used: (A) perturb the boundary vertices of the Wu--Deng triangulation to see if we can obtain a collinearity free counterexample; (B) enumerate small combinatorial types of topological disk triangulations and test generic realizations of each; (C) use the reduction ideas of \cite{JK:05} to filter out certain sets of triangulations (see Section \ref{sec:reductions}); (D) try to transplant the structural features of the known counterexamples in other degrees \cite{MS:75b,D:90,SSY:20} to $(r,d)=(1,3)$. The strategies (A) and (D) were not particularly helpful.

The final strategy which allowed the triangulation of Theorem~\ref{thm:triangulation} to be found was a combination of the reduction results of Section \ref{sec:reductions} below which allowed us to reduce the space of tested triangulations along with concentrated searches inside certain symmetric families of triangulations (e.g. those which are invariant under the action of $D_3$). 

\subsection{Reducing triangulations}\label{sec:reductions}

The search space could be significantly reduced using a theorem of Jakli\v{c} and Kozak \cite{JK:05}. The key idea is that certain boundary vertices of a triangulation and their stars can be deleted without affecting whether the dimension formula holds. We state this result in their notation. Let $T_0$ be a boundary vertex of $\calT$ of degree $s+1$, so that $T_0$ lies in $s$ triangles, let $\calT_1\subseteq\calT$ be a subtriangulation that includes all the triangles with the vertex $T_0$, and let $\calT\setminus\calT_1$ be the rest of $\calT$. We let $|\calT_1|$ denote the number of triangles of $\calT_1$. A triangulation with exactly one interior vertex is called a \emph{cell}. 

\begin{theorem}[{\cite[Theorem~1]{JK:05}}]\label{thm:red}
Suppose that $\dim S^1_3(\calT\setminus\calT_1)$ is equal to the lower bound \eqref{eq:LB3}. Further suppose there are no consecutive pairs of collinear edges on the common boundary of $\calT_1$ and $\calT\setminus\calT_1$, and that one of the following conditions is fulfilled:
\begin{enumerate}
\item[(1)] $s\le4$ and $|\calT_1|=s$;
\item[(2)] $s=2,3$, $|\calT_1|=s+2$, and $\calT_1$ includes a cell.
\end{enumerate}
Then $\dim S^1_3(\calT)$ is equal to the lower bound \eqref{eq:LB3}.
\end{theorem}

\begin{corollary}[{\cite[Corollary~2]{JK:05}}]\label{cor:red}
Let there be no collinearities of the edges at the interior vertices of $\calT$. Let the degrees of the interior vertices be at most $6$ and let the degrees of the vertices on the outer face be at most $5$. Then $\dim S^1_3(\calT)=\operatorname{LB}(\calT)$.
\end{corollary}

From Corollary~\ref{cor:red}, we see that a counterexample without collinear edge pairs must have an interior vertex of degree at least seven or a boundary vertex of degree at least six. The triangulation, $\widehat{\mathcal{T}}$, of Theorem~\ref{thm:triangulation} has both of these properties.

Before explaining how Theorem~\ref{thm:red} was used, we also note the following result which was useful in computation. 
\begin{definition}[Chord]
    A \emph{chord} is an interior edge with two boundary vertices.
\end{definition}

\begin{lemma}\label{lem:chord}
For a triangulation $\calT$, let
\begin{equation}\label{eq:defect}
\delta(\calT):=\dim S^1_3(\calT)-\operatorname{LB}(\calT)
=\dim\operatorname{coker}\big(\operatorname{div}:\bm V_h\to Q_h\big)-\sigma\ \ge0
\end{equation}
be the \emph{defect} of $\calT$, the amount by which the true dimension exceeds the lower bound \eqref{eq:LB3}. Let $e$ be a chord in $\calT$. Then $e$ splits $\calT$ into two triangulations $\calT_1,\calT_2$ of simply connected polygonal domains with $\delta(\calT)=\delta(\calT_1)+\delta(\calT_2)$.
\end{lemma}

\begin{proof}
The restriction map $S^1_3(\calT)\to S^1_3(\calT_1)$ is surjective. That is, the value and the normal derivative of a spline on $\calT_1$ along $e$ are a cubic and a quadratic function on $e$, which are also the value and normal derivative along $e$ of some cubic polynomial. This polynomial, taken on all of $\calT_2$, completes the spline over the whole triangulation $\calT$. On the other hand, the kernel of the restriction map consists of the splines on $\calT_2$ vanishing together with their gradient along $e$, which form a subspace of codimension $7$ of $S^1_3(\calT_2)$. Hence, $\dim S^1_3(\calT)=\dim S^1_3(\calT_1)+\dim S^1_3(\calT_2)-7$. 

Note that $e$ must be the only interior edge of $\calT$ that is a boundary edge of $\calT_1$ or $\calT_2$. Moreover, the end points of $e$ are boundary vertices of all three triangulations, and every other vertex has the same star in $\calT$ as in $\calT_1$ or $\calT_2$. Hence, $F=F_1+F_2$, $E_0=E_{0}(\calT_1)+E_{0}(\calT_2)+1$, $V_0=V_{0}(\calT_1)+V_{0}(\calT_2)$ and $\sigma=\sigma_1+\sigma_2$, so that $\operatorname{LB}(\calT)=\operatorname{LB}(\calT_1)+\operatorname{LB}(\calT_2)-7$ by
\eqref{eq:LB3}.
\end{proof}

We now explain how Theorem~\ref{thm:red} and Lemma~\ref{lem:chord} were used in the search process. 
\begin{definition}[Reducible triangulation]
    We call a triangulation $\calT$ \emph{reducible} if a sequence of deletions of stars of boundary vertices which belong to at most four triangles and without chords, ends with a single triangle.
\end{definition}
Note that the dimension formula \eqref{eq:LB} holds on a single triangle. Then, for any  reducible triangulation without collinear edge pairs, we can apply case (1) of Theorem~\ref{thm:red} along the reduction sequence in reverse order to conclude that the dimension formula holds.  In other words, suppose $\calT$ is a counterexample to the formula \eqref{eq:LB} without collinear edge pairs. Then, alternately deleting boundary vertices as long as case (1) allows and cutting along chords (via Lemma~\ref{lem:chord}) produces a smaller counterexample with zero chords and every boundary vertex lying in at least five triangles. Consequently, if a counterexample without collinear edge pairs exists, one exists with these two properties. With this in mind, the search space of candidate combinatorial types was restricted to such types from the outset. An illustration of the reduction procedure is given in Figure~\ref{fig:reduction}.

\begin{figure}[htb]
\centering
\begin{tikzpicture}[scale=0.6,every node/.style={font=\footnotesize,inner sep=1.5pt}]
  \fill[black!12] (0.00,0.00) -- (3.00,-1.25) -- (6.00,0.00) -- cycle;
  \fill[black!12] (0.00,0.00) -- (3.30,-2.70) -- (3.00,-1.25) -- cycle;
  \fill[black!12] (3.00,-1.25) -- (3.30,-2.70) -- (6.00,0.00) -- cycle;
  \draw[thin] (3.01,0.61) -- (3.47,0.95);
  \draw[thin] (3.01,0.61) -- (2.59,1.01);
  \draw[thin] (3.01,0.61) -- (0.00,0.00);
  \draw[thin] (3.01,0.61) -- (6.00,0.00);
  \draw[thin] (3.47,0.95) -- (2.59,1.01);
  \draw[thin] (3.47,0.95) -- (6.00,0.00);
  \draw[thin] (3.47,0.95) -- (3.22,1.71);
  \draw[thin] (2.59,1.01) -- (0.00,0.00);
  \draw[thin] (2.59,1.01) -- (3.22,1.71);
  \draw[thin] (2.59,1.01) -- (2.28,1.34);
  \draw[thin] (0.00,0.00) -- (6.00,0.00);
  \draw[thin] (0.00,0.00) -- (2.28,1.34);
  \draw[thin] (0.00,0.00) -- (2.42,2.11);
  \draw[very thick] (0.00,0.00) -- (3.00,5.20);
  \draw[thin] (0.00,0.00) -- (3.00,-1.25);
  \draw[very thick] (0.00,0.00) -- (3.30,-2.70);
  \draw[thin] (6.00,0.00) -- (3.22,1.71);
  \draw[thin] (6.00,0.00) -- (3.00,5.20);
  \draw[thin] (6.00,0.00) -- (3.73,2.33);
  \draw[very thick] (6.00,0.00) -- (7.62,4.40);
  \draw[thin] (6.00,0.00) -- (6.36,2.17);
  \draw[thin] (6.00,0.00) -- (5.94,2.91);
  \draw[thin] (6.00,0.00) -- (5.49,3.69);
  \draw[thin] (6.00,0.00) -- (5.06,4.43);
  \draw[thin] (6.00,0.00) -- (3.00,-1.25);
  \draw[very thick] (6.00,0.00) -- (3.30,-2.70);
  \draw[thin] (3.22,1.71) -- (2.28,1.34);
  \draw[thin] (3.22,1.71) -- (2.42,2.11);
  \draw[thin] (3.22,1.71) -- (3.73,2.33);
  \draw[thin] (3.22,1.71) -- (3.25,2.73);
  \draw[thin] (3.22,1.71) -- (2.97,2.68);
  \draw[thin] (2.28,1.34) -- (2.42,2.11);
  \draw[thin] (2.42,2.11) -- (3.00,5.20);
  \draw[thin] (2.42,2.11) -- (2.97,2.68);
  \draw[thin] (3.00,5.20) -- (3.73,2.33);
  \draw[thin] (3.00,5.20) -- (3.25,2.73);
  \draw[thin] (3.00,5.20) -- (2.97,2.68);
  \draw[very thick] (3.00,5.20) -- (7.62,4.40);
  \draw[thin] (3.00,5.20) -- (5.06,4.43);
  \draw[thin] (3.73,2.33) -- (3.25,2.73);
  \draw[thin] (3.25,2.73) -- (2.97,2.68);
  \draw[thin] (7.62,4.40) -- (6.36,2.17);
  \draw[thin] (7.62,4.40) -- (5.94,2.91);
  \draw[thin] (7.62,4.40) -- (5.49,3.69);
  \draw[thin] (7.62,4.40) -- (5.06,4.43);
  \draw[thin] (6.36,2.17) -- (5.94,2.91);
  \draw[thin] (5.94,2.91) -- (5.49,3.69);
  \draw[thin] (5.49,3.69) -- (5.06,4.43);
  \draw[thin] (3.00,-1.25) -- (3.30,-2.70);
  \fill (3.01,0.61) circle (1.5pt);
  \fill (3.47,0.95) circle (1.5pt);
  \fill (2.59,1.01) circle (1.5pt);
  \fill (0.00,0.00) circle (1.5pt);
  \fill (6.00,0.00) circle (1.5pt);
  \fill (3.22,1.71) circle (1.5pt);
  \fill (2.28,1.34) circle (1.5pt);
  \fill (2.42,2.11) circle (1.5pt);
  \fill (3.00,5.20) circle (1.5pt);
  \fill (3.73,2.33) circle (1.5pt);
  \fill (3.25,2.73) circle (1.5pt);
  \fill (2.97,2.68) circle (1.5pt);
  \fill (7.62,4.40) circle (1.5pt);
  \fill (6.36,2.17) circle (1.5pt);
  \fill (5.94,2.91) circle (1.5pt);
  \fill (5.49,3.69) circle (1.5pt);
  \fill (5.06,4.43) circle (1.5pt);
  \fill (3.00,-1.25) circle (1.5pt);
  \fill (3.30,-2.70) circle (1.5pt);
  \node[left] at (0.00,0.00) {$c_0$};
  \node[below right] at (6.00,0.00) {$c_1$};
  \node[above] at (3.00,5.20) {$c_2$};
  \node[below left] at (3.00,-1.25) {$x_1$};
  \node[below] at (3.30,-2.70) {$x_2$};
  \node[right] at (7.62,4.40) {$w$};
  \node at (4.2,-3.6) {(a)};
\end{tikzpicture}
\hspace{1cm}
\begin{tikzpicture}[scale=0.6,every node/.style={font=\footnotesize,inner sep=1.5pt}]
  \draw[thin] (3.01,0.61) -- (3.47,0.95);
  \draw[thin] (3.01,0.61) -- (2.59,1.01);
  \draw[thin] (3.01,0.61) -- (0.00,0.00);
  \draw[thin] (3.01,0.61) -- (6.00,0.00);
  \draw[thin] (3.47,0.95) -- (2.59,1.01);
  \draw[thin] (3.47,0.95) -- (6.00,0.00);
  \draw[thin] (3.47,0.95) -- (3.22,1.71);
  \draw[thin] (2.59,1.01) -- (0.00,0.00);
  \draw[thin] (2.59,1.01) -- (3.22,1.71);
  \draw[thin] (2.59,1.01) -- (2.28,1.34);
  \draw[very thick] (0.00,0.00) -- (6.00,0.00);
  \draw[thin] (0.00,0.00) -- (2.28,1.34);
  \draw[thin] (0.00,0.00) -- (2.42,2.11);
  \draw[very thick] (0.00,0.00) -- (3.00,5.20);
  \draw[thin] (6.00,0.00) -- (3.22,1.71);
  \draw[thin] (6.00,0.00) -- (3.73,2.33);
  \draw[very thick] (6.00,0.00) -- (7.62,4.40);
  \draw[thin] (6.00,0.00) -- (6.36,2.17);
  \draw[thin] (6.00,0.00) -- (5.94,2.91);
  \draw[thin] (6.00,0.00) -- (5.49,3.69);
  \draw[thin] (6.00,0.00) -- (5.06,4.43);
  \draw[thin] (3.22,1.71) -- (2.28,1.34);
  \draw[thin] (3.22,1.71) -- (2.42,2.11);
  \draw[thin] (3.22,1.71) -- (3.73,2.33);
  \draw[thin] (3.22,1.71) -- (3.25,2.73);
  \draw[thin] (3.22,1.71) -- (2.97,2.68);
  \draw[thin] (2.28,1.34) -- (2.42,2.11);
  \draw[thin] (2.42,2.11) -- (3.00,5.20);
  \draw[thin] (2.42,2.11) -- (2.97,2.68);
  \draw[thin] (3.00,5.20) -- (3.73,2.33);
  \draw[thin] (3.00,5.20) -- (3.25,2.73);
  \draw[thin] (3.00,5.20) -- (2.97,2.68);
  \draw[very thick] (3.00,5.20) -- (7.62,4.40);
  \draw[thin] (3.00,5.20) -- (5.06,4.43);
  \draw[thin] (3.73,2.33) -- (3.25,2.73);
  \draw[thin] (3.25,2.73) -- (2.97,2.68);
  \draw[thin] (7.62,4.40) -- (6.36,2.17);
  \draw[thin] (7.62,4.40) -- (5.94,2.91);
  \draw[thin] (7.62,4.40) -- (5.49,3.69);
  \draw[thin] (7.62,4.40) -- (5.06,4.43);
  \draw[thin] (6.36,2.17) -- (5.94,2.91);
  \draw[thin] (5.94,2.91) -- (5.49,3.69);
  \draw[thin] (5.49,3.69) -- (5.06,4.43);
  \draw[very thick,dashed] (6.00,0.00) -- (3.00,5.20);
  \draw[thin,dashed,black!60] (0.00,0.00) -- (3.00,-1.25) -- (6.00,0.00) -- cycle;
  \draw[thin,dashed,black!60] (0.00,0.00) -- (3.30,-2.70) -- (3.00,-1.25) -- cycle;
  \draw[thin,dashed,black!60] (3.00,-1.25) -- (3.30,-2.70) -- (6.00,0.00) -- cycle;
  \fill (3.01,0.61) circle (1.5pt);
  \fill (3.47,0.95) circle (1.5pt);
  \fill (2.59,1.01) circle (1.5pt);
  \fill (0.00,0.00) circle (1.5pt);
  \fill (6.00,0.00) circle (1.5pt);
  \fill (3.22,1.71) circle (1.5pt);
  \fill (2.28,1.34) circle (1.5pt);
  \fill (2.42,2.11) circle (1.5pt);
  \fill (3.00,5.20) circle (1.5pt);
  \fill (3.73,2.33) circle (1.5pt);
  \fill (3.25,2.73) circle (1.5pt);
  \fill (2.97,2.68) circle (1.5pt);
  \fill (7.62,4.40) circle (1.5pt);
  \fill (6.36,2.17) circle (1.5pt);
  \fill (5.94,2.91) circle (1.5pt);
  \fill (5.49,3.69) circle (1.5pt);
  \fill (5.06,4.43) circle (1.5pt);
  \node[left] at (0.00,0.00) {$c_0$};
  \node[below right] at (6.00,0.00) {$c_1$};
  \node[above] at (3.00,5.20) {$c_2$};
  \node at (4.2,-3.6) {(b)};
\end{tikzpicture}
\hspace{1cm}
\begin{tikzpicture}[scale=0.6,every node/.style={font=\footnotesize,inner sep=1.5pt}]
  \draw[thin] (3.01,0.61) -- (3.47,0.95);
  \draw[thin] (3.01,0.61) -- (2.59,1.01);
  \draw[thin] (3.01,0.61) -- (0.00,0.00);
  \draw[thin] (3.01,0.61) -- (6.00,0.00);
  \draw[thin] (3.47,0.95) -- (2.59,1.01);
  \draw[thin] (3.47,0.95) -- (6.00,0.00);
  \draw[thin] (3.47,0.95) -- (3.22,1.71);
  \draw[thin] (2.59,1.01) -- (0.00,0.00);
  \draw[thin] (2.59,1.01) -- (3.22,1.71);
  \draw[thin] (2.59,1.01) -- (2.28,1.34);
  \draw[very thick] (0.00,0.00) -- (6.00,0.00);
  \draw[thin] (0.00,0.00) -- (2.28,1.34);
  \draw[thin] (0.00,0.00) -- (2.42,2.11);
  \draw[very thick] (0.00,0.00) -- (3.00,5.20);
  \draw[thin] (6.00,0.00) -- (3.22,1.71);
  \draw[very thick] (6.00,0.00) -- (3.00,5.20);
  \draw[thin] (6.00,0.00) -- (3.73,2.33);
  \draw[thin] (3.22,1.71) -- (2.28,1.34);
  \draw[thin] (3.22,1.71) -- (2.42,2.11);
  \draw[thin] (3.22,1.71) -- (3.73,2.33);
  \draw[thin] (3.22,1.71) -- (3.25,2.73);
  \draw[thin] (3.22,1.71) -- (2.97,2.68);
  \draw[thin] (2.28,1.34) -- (2.42,2.11);
  \draw[thin] (2.42,2.11) -- (3.00,5.20);
  \draw[thin] (2.42,2.11) -- (2.97,2.68);
  \draw[thin] (3.00,5.20) -- (3.73,2.33);
  \draw[thin] (3.00,5.20) -- (3.25,2.73);
  \draw[thin] (3.00,5.20) -- (2.97,2.68);
  \draw[thin] (3.73,2.33) -- (3.25,2.73);
  \draw[thin] (3.25,2.73) -- (2.97,2.68);
  \fill (3.01,0.61) circle (1.5pt);
  \fill (3.47,0.95) circle (1.5pt);
  \fill (2.59,1.01) circle (1.5pt);
  \fill (0.00,0.00) circle (1.5pt);
  \fill (6.00,0.00) circle (1.5pt);
  \fill (3.22,1.71) circle (1.5pt);
  \fill (2.28,1.34) circle (1.5pt);
  \fill (2.42,2.11) circle (1.5pt);
  \fill (3.00,5.20) circle (1.5pt);
  \fill (3.73,2.33) circle (1.5pt);
  \fill (3.25,2.73) circle (1.5pt);
  \fill (2.97,2.68) circle (1.5pt);
  \node[left] at (0.00,0.00) {$c_0$};
  \node[below right] at (6.00,0.00) {$c_1$};
  \node[above] at (3.00,5.20) {$c_2$};
  \node at (4.2,-1.6) {(c)};
  \node at (3.0,-0.7) {$C$};
  \draw[very thick] (6.95,0.55) -- (3.95,5.75);
  \draw[very thick] (6.95,0.55) -- (8.57,4.95);
  \draw[thin] (6.95,0.55) -- (7.32,2.72);
  \draw[thin] (6.95,0.55) -- (6.89,3.46);
  \draw[thin] (6.95,0.55) -- (6.44,4.24);
  \draw[thin] (6.95,0.55) -- (6.02,4.98);
  \draw[very thick] (3.95,5.75) -- (8.57,4.95);
  \draw[thin] (3.95,5.75) -- (6.02,4.98);
  \draw[thin] (8.57,4.95) -- (7.32,2.72);
  \draw[thin] (8.57,4.95) -- (6.89,3.46);
  \draw[thin] (8.57,4.95) -- (6.44,4.24);
  \draw[thin] (8.57,4.95) -- (6.02,4.98);
  \draw[thin] (7.32,2.72) -- (6.89,3.46);
  \draw[thin] (6.89,3.46) -- (6.44,4.24);
  \draw[thin] (6.44,4.24) -- (6.02,4.98);
  \fill (6.95,0.55) circle (1.5pt);
  \fill (3.95,5.75) circle (1.5pt);
  \fill (8.57,4.95) circle (1.5pt);
  \fill (7.32,2.72) circle (1.5pt);
  \fill (6.89,3.46) circle (1.5pt);
  \fill (6.44,4.24) circle (1.5pt);
  \fill (6.02,4.98) circle (1.5pt);
  \node[right] at (8.57,4.95) {$w$};
  \node at (7.70,0.25) {$B$};
\end{tikzpicture}
\caption{The reduction procedure. (a) A triangulation $\calT$. The boundary vertex $x_2$ lies in two triangles and its
neighbour $x_1$ is interior, so $x_2$ can be deleted by case (1) of Theorem~\ref{thm:red}. Afterwards $x_1$ lies in a
single triangle and can be deleted too (shaded). (b) Now every boundary vertex lies in at least five triangles, so no further deletion is possible, but $c_1c_2$ (dashed) is a chord. (c) Cutting along the chord (Lemma~\ref{lem:chord}) leaves two pieces, $B$ and $C$. $B$ is reducible, so the dimension formula holds on $B$, and so $\delta(\calT)=\delta(C)$. $C$ has no chord and each of its three boundary vertices lies in five triangles. Hence, $\delta(C)$ is all that needs to be computed, and if positive, both $C$ and $\mathcal{T}$ are counterexamples.}
\label{fig:reduction}
\end{figure}
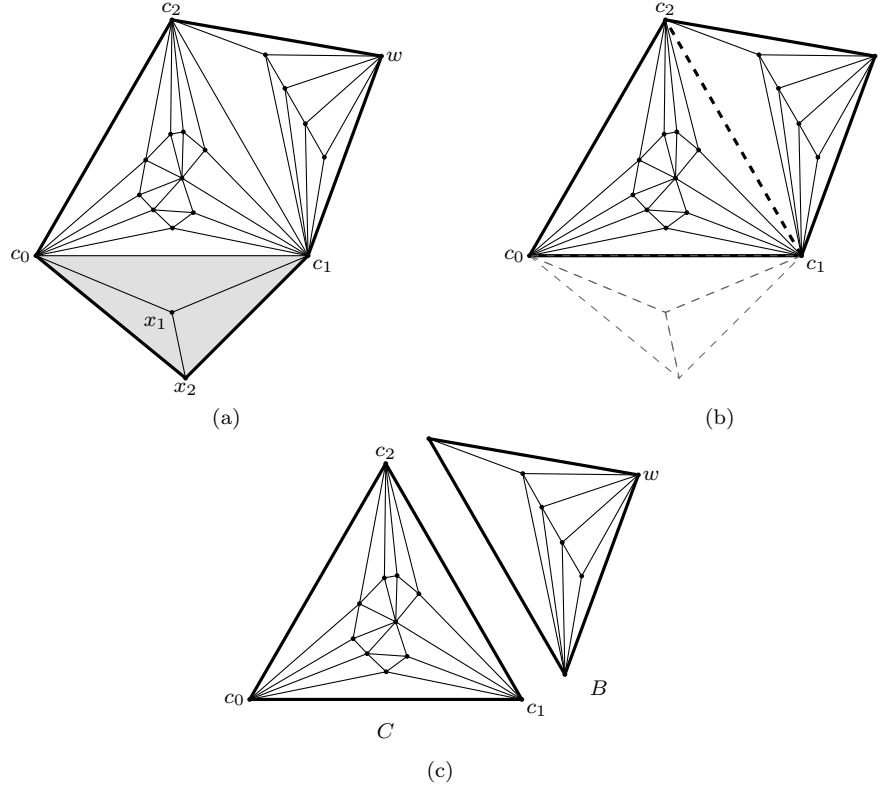

Similar reduction ideas are present in the manuscript of Morgan and Scott \cite[Theorem~3]{MS:96}. The quadratic analogue, with $s\le2$, is \cite[Theorem~5.7]{RB:26}. In particular, the reduction procedure for quadratics discussed in \cite{RB:26} allows one to remove two triangles at a time which have a shared boundary vertex. This is exactly equivalent to the notion of elementary strong collapse present in \cite{BM:12}. It would be interesting if there are further connections between collapse-type ideas present in combinatorial topology and splines.

\section*{AI Usage Statement}
Claude Fable 5.1 was used to perform the search for the triangulation as discussed in Section~\ref{sec:search_procedure}. Claude Fable 5.1 was also used to generate the tables and figures in this paper and help write code to perform the relevant computations. The entire article aside from this was written by the author. All code was verified by the author who takes full responsibility for the content of the article.

\section*{Acknowledgments}
The author thanks Johnny Guzm\'an for many discussions on Strang's conjecture and for reviewing an earlier version of this article.
\bibliographystyle{siam}
\bibliography{ref}

\end{document}